\documentclass{article}      
\usepackage{amssymb}
\usepackage{amsmath}
\usepackage{latexsym}
\usepackage[mathscr]{euscript}
\usepackage{graphicx}
\usepackage{amscd}
\usepackage{amsthm} 
\usepackage{fullpage} 
\usepackage{wrapfig}
\newtheorem{theorem}{Theorem}

\newtheorem{lemma}[theorem]{Lemma}  
\newtheorem{proposition}[theorem]{Proposition}

\newcommand{\bc}{\mathbb{C}}
\newcommand{\bp}{\mathbb{ P}}
\newcommand{\bz}{\mathbb{Z}}
\newcommand{\br}{\mathbb{R}}
\newcommand{\bq}{\mathbb{Q}}

\newcommand{\p}{\partial}
\newcommand{\cc}{\mathcal{C}}

\newcommand{\cp}{\mathcal{P}}

\newcommand{\hk}{\hookrightarrow}

\newcommand{\med}{\medskip}
\newcommand{\la}{\longrightarrow}
\newcommand{\bfl}{\begin{flushleft}}
\newcommand{\efl}{\end{flushleft}}

\newcommand{\eps}{\epsilon}

\newcommand{\xr}{\xrightarrow}

\newcommand{\bcp}{\bc \bp}

\newcommand{\sgn}{\mathscr S_{g,n}}
\newcommand{\crs}{\mathscr{S}}
\newcommand{\sgnm}{\mathscr S_{g,n,m}}
\newcommand{\fgnm}{F_{g,n,m}}
\newcommand{\dgnm}{Diff(\fgnm; \p)}

\newcommand{\sgo}{\mathscr S_g^1}
\newcommand{\sgok}{\mathscr S_{g,0,1}(K)}
\newcommand{\crgk}{\crs_g(K)}
\newcommand{\G}{\Gamma}
\newcommand{\gp}{\G^1_g}
\newcommand{\gb}{\G_{g,1}}

 \newcommand{\ock}{\Omega^{\infty-1}_0(\bcp^\infty_{-1} \wedge K_+)}
 \newcommand{\ockz}{\Omega^{\infty}(\bcp^\infty_{-1} \wedge K_+)}

\begin{document}

  \title{Stability for closed surfaces in a background space}  
  \author{Ralph L. Cohen \thanks{The first author was partially supported by  NSF grants 0603713 and 0905809 and ERC-AdG 228082 } \\ Department of Mathematics \\Stanford University \\ Bldg. 380 \\ Stanford, CA 94305, USA \and   Ib Madsen \thanks{The second author was supported by ERC-AdG 228082} \\  Department of Mathematical Sciences \\  University of  Copenhagen \\ Universitetsparken 5 \\DK-2100 Copenhagen, Denmark }
\date{\today}
\maketitle  
 \begin{abstract}  In this paper we present a new proof of the homological stability of the moduli space of closed surfaces in a simply connected background space $K$, which we denote by $\crs_g (K)$.    The homology stability of surfaces in $K$ with an arbitrary number of boundary components, $\sgn (K)$ was studied by the authors in \cite{cohenmadsen}.  The study there relied on stability  results for the homology of mapping class groups, $\G_{g,n}$ with certain families of twisted coefficients.   It turns out that these mapping class groups only have homological stability when $n$, the number of boundary components,  is positive, or in the closed case   when the coefficient modules are trivial.  Because of this we present a new proof of the rational homological stability for $\crs_g(K)$, that is homotopy theoretic in nature.  We also take the opportunity to prove a new stability theorem for closed
 surfaces in $K$ that have marked points.    \end{abstract}

 \tableofcontents

 \section*{Introduction} In \cite{cohenmadsen}, the authors studied stability properties for moduli spaces of surfaces in a simply connected background space, $K$.  These moduli spaces, denoted $\sgn (K)$,  consist of surfaces $S_{g,n}$ of genus $g$ with $n$ parameterized boundary components, together with a map
 $f : S_{g,n} \to K$ which restricts to  the boundary $\p S_{g,n}$ in  a prescribed way.  As observed in \cite{cohenmadsen} the  homotopy type of these moduli spaces don't depend on the choice of this boundary condition, so we will assume $f$ is constant on $\p S_{g,n}$, mapping it to a fixed basepoint $x_0 \in K$. The main results of that paper were Theorems 0.1 and 0.3 in \cite{cohenmadsen}  which together  identify the ``stable homology",  $H_q(\sgn (X))$    for $2q+4 \leq g$.   
 
 This result was proved using spectral sequence and Postnikov tower arguments, whose main input was a calculation of the stable homology of mapping class groups with certain families of twisted coefficients (``coefficient systems of finite degree"),  $H_*(\G_{g,n}; V_{g,n})$.   Here $\G_{g,n}$ is the mapping class group of  orientation preserving diffeomorphisms of a fixed surface $F_{g,n}$ of genus $g$ and $n$-boundary
 components, that fix the boundary pointwise, $\G_{g,n} = \pi_0 (Diff(F_{g,n}, \p F_{g,n}))$. 
 
 The goals of the present note are threefold.  First, we wish to describe an error  in \cite{cohenmadsen} in the group homology calculation for closed surfaces ($n = 0$).  Theorem 0.4  of that paper claims that     for coefficient systems of appropriately finite degree $d$, then  $H_q(\G_{g,n}, V_{g,n})$ is independent of $g$ and $n$, so long as $2q+d+2 \leq g-1.$  This theorem is true, and the proof in \cite{cohenmadsen}
 is correct, so long as $n > 0$.  However, for closed surfaces, $n=0$, this theorem is not true, unless the
 the coefficient system has degree 0 (i.e is constant), which in turn was proved by Harer \cite{harer}, with this
 improved stability range proved by Ivanov \cite{ivanov}.  A counter example to this theorem in the closed case is given by a calculation of S. Morita \cite{morita}, which was pointed out to us by J. Ebert.
 
 The second goal of this paper is to give a new proof of the homological stability theorem for $\crs_g(K)$, for closed surfaces.  This proof is purely homotopy theoretic in nature, using homological stability for $\sgn (K)$ for $n >0$,  and some basic, well known relationships between  diffeomorphism groups of closed surfaces, of surfaces with
 boundary, and those with marked points.     We also make heavy use, as we did in \cite{cohenmadsen},  of the theorem of Madsen and Weiss \cite{madsenweiss} proving a generalization of a conjecture of Mumford, that can be viewed as establishing these stability results when $X$ is a point.  
 
 Along the way to giving this proof, we prove a new theorem,  giving a stability result for closed surfaces  with marked points (see Theorems \ref{markedpoint} and \ref{manymarked} below).  This generalizes results of B\"odigheimer and Tillmann \cite{bodtill}.

 The third and last goal of the paper is to use the homotopy theoretic techniques mentioned above to give a new proof of Morita's calculation of $H_1(\Gamma_g, H_1(F_g))$ which demonstrates the lack of stability for these homology groups.   
 
 We remark that since the writing of \cite{cohenmadsen}, results there have been generalized in two different ways.  In his thesis \cite{boldsen},  S. Boldsen significantly improved the stability range of the homology of mapping class groups, both in the setting of twisted coefficients (when the surface has at least one boundary), and in the setting of closed surfaces, when the coefficients are trivial.    As we will point out below, Boldsen's improved stability range allows for a similar improvement of the stability range for $H_*(\sgn (X))$, $n \geq 0$.  
 Secondly,  in \cite{oscar},  O. Randal-Williams generalized our result for the homological stability of $\sgn (K)$,  
 by proving a stability theorem for surfaces having a general tangential structure.  He also obtained an improved stable range.

 Because of the work of Randal-Williams, the truth of the homological stability for closed surfaces, $\crs_g(K)$ is not in question.  However we take the opportunity in this note to give a new, short  proof  of this theorem when one assumes   rational coefficients (or coefficients in any field of characteristic zero).  This proof is quite easy,
 and shows the relevance of the Becker-Gottlieb transfer map.

 \med
 This paper is organized as follows.
 In section one we review the construction and relationships between the moduli spaces $\sgn (K)$. We then prove the main new theorem of this paper, Theorem \ref{markedpoint},  which gives homological stability of
closed surface spaces with marked points.  In section 2 we describe a relationship with the Becker-Gottlieb transfer map,
and give a short proof of the homological stability of $\crs_g(K)$.  In section 3 we give a new proof of Morita's calculation \cite{morita}.  

\med
The authors would like to thank J. Ebert for originally pointing out the error in \cite{cohenmadsen}.  The first author would also like to thank the Department of Mathematical Sciences at the University of Copenhagen   for its hospitality while this work was being carried out.

 \section{Spaces of surfaces, and stability of closed surfaces with marked points}
 \subsection{Spaces of surfaces and their relationships}
   For each $n \geq 0$,  let $c_n \subset \br^\infty$ be the image of a fixed embedding $e_n : \coprod_n S^1 \hk \br^\infty$.  
 Let $K$ be a simply connected space with basepoint $x_0 \in K$.  We define surface spaces $\sgnm (K)$  much like in \cite{cohenmadsen}:

\begin{align}
\sgnm (K) = &\{(S_{g,n,m}, t, \phi,  f): \,  \text{where $t>0$,  $S_{g,n,m} \subset \br^\infty \times [0,t]$ is a smooth oriented surface of genus $g$ with } \notag  \\
 &\text{ $n+m$ boundary components, $\phi : \coprod_{n+m} S^1 \xr{\cong} \p S$ is a parameterization of the boundary,}\notag \\
& \text{   and $f : S_{g,n,m} \to K$ is a continuous map whose restriction to the boundary is constant} \notag \\
&\text{at the basepoint,  $\p f : \p S_{g,n,m} \to x_0 \in K.$  } \} \notag
\end{align}
In this description,  $S_{g,n,m} \subset \br^\infty \times [0,t]$ is embedded ``neatly" as in \cite{madsenweiss}.  In particular
 the boundary, $\p S$, is partitioned
 $$
 \p S_{g,n,m} = \p_{in}S_{g,n,m}  \sqcup  \p_{out}S_{g,n,m}
 $$
 where the ``incoming"  boundary  $\p_{in}S_{g,n,m}  = S_{g,n,m} \cap (\br^\infty \times \{0\})$ is equal to $c_n$, and the ``outgoing" boundary $\p_{out}S_{g,n,m}  = S_{g,n,m} \cap (\br^\infty \times \{t\})$  is equal to $c_m$.      The parameterization $\phi$ is an orientation preserving diffeomorphism.        
 
Now let $F_{g,n,m}$ be a fixed smooth, oriented surface of genus $g$ with $n+m$ parameterized boundary components;  $n$ of these boundary components are designated as ``incoming", and the remaining $m$ are designated as ``outgoing".  Let $Diff (F_{g,n,m}, \p)$ denote the group of orientation preserving diffeomorphisms  that fix the boundary pointwise. As described in \cite{madsenweiss}, \cite{GMTW}, \cite{cohenmadsen},  the spaces $\sgnm (K)$ are topologized so as to give  homeomorphisms \begin{equation}
\sgnm (K)  \cong    \br  \times  Emb_\p(F_{g, n,m}, \br^\infty \times [0,1]) \times_{\dgnm} Map ((\fgnm, \p), (K, x_0)).
\end{equation}\label{bdry}
Here $Emb_\p(F_{g, n,m}, \br^\infty \times [0,1])$ denotes the space of neat embeddings that extend $e_n : \coprod_n S^1 \to \br^\infty \times \{0\}$ on $\p_{in} F_{g, n,m}$, and $e_m : \coprod_m S^1 \to \br^\infty \times \{1\}$ on $\p_{out} F_{g, n,m}$.  Since this embedding space is contractible with a free action of $\dgnm$, this gives a  homotopy equivalence  
\begin{equation}\label{bound}
 \sgnm (K) \simeq E\dgnm \times_{\dgnm} Map ((\fgnm, \p), (K, x_0)).
\end{equation}
For $n=m=0$, we ease notation by deleting the subscripts.  So \begin{equation}\label{closed}\crs_g(K) \cong \br  \times   Emb (F_g, \br^\infty) \times_{Diff (F_g)} Map (F_g, K).  
\end{equation}

Finally we define the space $\sgo (K)$ to be the space of closed surfaces in the background space $K$ that have a marked point.  In other words,
 
  \begin{align}\label{mark}
\sgo (K) =  &\{(S_g, t, x,  f): \,  \text{where $t>0$,  $S_{g} \subset \br^\infty \times (0,t)$ is a smooth oriented surface of genus $g$, } \notag  \\
 &x \in S_g, \, \text{and  \, $f : S_g \to K$ is a continuous map}\}  \, \notag \\
& \cong \br  \times   Emb (F_g, \br^\infty) \times_{Diff (F_g)} (F_g \times Map (F_g, K)).
\end{align}

We now consider the relationships between the spaces, $\crs_{g, 0,1}$,  $\sgo$, and $\crs_g$.   First consider the map
  $p : \sgo (K) \to \crs_g (K)$  given by forgetting the marked point.  This is the ``universal curve".  Then by (\ref{closed}) and (\ref{mark})  we have the following:

\med
\begin{proposition}\label{markfib} The map $p :  \sgo (K) \to \crs_g (K)$ is a fiber bundle with fiber equal to the surface $F_g$.
\end{proposition}
\med

Consider the fiberwise, or vertical tangent bundle $T_{v}\sgo(K)$.  This is an oriented, two-dimensional vector bundle classified
by a map $\tau_v: \sgo (K) \to \bcp^\infty$.  Concretely, this can be defined as follows (following \cite{madsenweiss} and \cite{GMTW}).
\begin{align}\label{teevee}
\tau_v : \sgo (K) &\to \bcp^\infty   \\
(S_g, t, x,  f)  &\to T_xS_g \subset \br^\infty  \notag \\  \notag
\end{align}
Here we are thinking of $\bcp^\infty$ as the Grassmannian of oriented two dimensional subspaces of $\br^\infty$.  Now consider the evaluation map
\begin{align}\label{ee}
e : \sgo (K) \to K  \\
(S_g, t, x,  f)  &\to f(x) \notag 
\end{align}  

\begin{proposition} \label{bdrymark}  There is a homotopy fibration sequence
$$
\sgok \xr{\iota} \sgo (K) \xr{ \tau_v \times e} \bcp^\infty \times K.
$$
   \end{proposition}

\begin{proof} Consider a fixed, neat embedding of the closed disk, $e : D^2 \subset \br^\infty \times [0, 1/2)$, whose boundary is the fixed embedding of the unit circle, $e_1 : S^1 \hk  \br^\infty \times \{0\}$.  Now let $(S_{g,0,1}, t, \phi,  f) \in \sgok$.   Using the parameterization of the boundary $\phi$, one can identify $\p S_{g,0,1}$ with  $c_1$, which is the image of the unit circle $e_1 :S^1 \hk \br^\infty$.  Let $S_g = S_{g,0,1}\cup_{S^1} D^2$.  This is the closed surface one
obtains by capping off the boundary $\p S_{g,0,1}$.  One can also extend the map $f : (S_{g,0,1}, \p S_{g,0,1}) \to (K, x_0)$ to $S_g$ by defining it to be constant (at the basepoint $x_0$) on $D^2$.   This construction defines a map
\begin{align}
 EDiff(F_{g,0,1}, \p) \times_{Diff(F_{g,0,1}, \p)} &Map ((F_{g,0,1}, \p),  (K, x_0))   \notag \\
 &\xr{q}  EDiff(F_{g}, y) \times_{Diff(F_{g}, y)} Map ((F_{g}, y),  (K, x_0)) \notag
 \end{align}
where  $y \in F_g$ is a marked point and $Diff (F_g, y)$  is the group of orientation preserving diffeomorphisms  that fix $y \in F_g$.     We shall need the well-known homotopy fibration sequence, (see Lemma \ref{fibration} below):
\begin{equation}\label{fibra}
BDiff(F_{g,0,1}, \p)\to BDiff (F_g, y) \xr{\tau_v}  \bcp^\infty
\end{equation}
where, if our model for the classifying space $BDiff (F_g, x)$ is given by $Emb(F_g, \br^\infty)/Diff (F_g, y) = \{(S_g, z): S_g \subset \br^\infty, \, z \in S_g\}$,
then $\tau_v (S_g, z) = T_z S_g \subset \br^\infty$.  Furthermore, since there is an obvious relative homotopy equivalence
$$
(F_{g,0,1}, \p F_{g,0,1}) \simeq  (F_g, y),
$$
we then have an induced equivalence between the mapping spaces
$$
 Map ((F_{g,0,1}, \p),  (K, x_0))  \simeq Map ((F_{g}, y),  (K, x_0))
 $$
 which is equivariant with respect to the homomorphism $Diff(F_{g,0,1}, \p)\to Diff (F_g, y)$.  Therefore there is an induced homotopy fibration sequence
\begin{align}
 EDiff(F_{g,0,1}, \p) \times_{Diff(F_{g,0,1}, \p)} &Map ((F_{g,0,1}, \p),  (K, x_0))    \xr{q}  \notag \\
  &EDiff(F_{g}, y) \times_{Diff(F_{g}, y)} Map ((F_{g}, y),  (K, x_0))  \xr{\tau_v} \bcp^\infty. \notag
\end{align}

Now the inclusion of the based maps into the unbased maps $ Map ((F_{g}, y),  (K, x_0)) \to Map (F_g, K)$ is equivariant with respect to the action
of $Diff(F_g, y)$.  Furthermore it is the inclusion of the fiber of the equivariant fibration
\begin{align}
Map(F_g, K) &\xr{e}K \notag \\
f &\to f(y) \notag
\end{align} where the action of $Diff(F_g, y)$ on $K$ is trivial.  Putting these fibrations together yields a homotopy fibration sequence

\begin{align}\label{fiber}
 EDiff(F_{g,0,1}, \p) \times_{Diff(F_{g,0,1}, \p)} &Map ((F_{g,0,1}, \p),  (K, x_0))    \xr{q} \\
  &EDiff(F_{g}, y) \times_{Diff(F_{g}, y)} Map (F_{g},  K)  \xr{\tau_v \times e } \bcp^\infty \times K. \notag
\end{align}
Now notice that $BDiff(F_g, y) \simeq EDiff(F_g)\times_{Diff(F_g)} F_g.$  This is seen by observing that the natural action of $Diff (F_g)$ on $F_g$ defines a homeomorphism of the  homogeneous
space $Diff(F_g)/Diff(F_g, y)$ with $F_g$.  More generally, if $X$ is any space with a $Diff (F_g)$-action,  there is an equivalence,
$$
EDiff (F_g) \times_{Diff(F_g)} (F_g \times X)  \simeq EDiff(F_g, y) \times_{Diff (F_g,y)} X.
$$
Thus fibration sequence (\ref{fiber}) becomes a homotopy fibration sequence,
\begin{align} 
 EDiff(F_{g,0,1}, \p) \times_{Diff(F_{g,0,1}, \p)} &Map ((F_{g,0,1}, \p),  (K, x_0))    \xr{q}  \notag \\
  &EDiff(F_{g}) \times_{Diff(F_{g})} Map (F_{g},  K)  \xr{\tau_v \times e } \bcp^\infty \times K. \notag
\end{align}
Using (\ref{bound}) and (\ref{mark}), this yields the homotopy fibration sequence,
$$
 \sgok \xr{\iota} \sgo (K) \xr{ \tau_v \times e} \bcp^\infty \times K.
$$ \end{proof}

The homotopy fibration (\ref{fibra})  that was used above is a special case of the following well-known lemma:

\begin{lemma}\label{fibration}  Let $M^d$ be a closed, oriented smooth manifold of dimension $d$, and $y \in M$ a marked point.  There is a homotopy fibration sequence
$$
Diff (M, D_\eps (y)) \to Diff (M, y) \xr{d} GL^+(d, \br)
$$
where $D_\eps(y)$ is a small open disk and the map $d$ denotes the differential at $y$.
\end{lemma}
\begin{proof}  We fix a chart $(\br^d, 0) \to (M, y)$ around $y$, and let $U_\eps (y)$ be a small open neighborhood of $y$ corresponding to an open $\eps$-disk $\br^d_\eps \subset \br^d$ around the origin.  The covering isotopy thoerem shows that we have a Serre fibration
$$
Diff (M, U_\eps (y)) \to Diff (M,y)  \xr{D} Emb((\br^d_\eps, 0),  (\br^d, 0)).
$$
It remains to show that this embedding space is homotopy equivalent to $GL^+(d, \br)$.   This follows because an embedding $f : (\br^d_\eps, 0) \hk (\br^d, 0)$ is homotopic to a linear map via the standard homotopy
$$
f_t(u) = \begin{cases}  \frac{f(tu)}{t},  \quad &t > 0, \\
df_0(u), \quad &t=0.
\end{cases}
$$
Since $Diff (M, D_\eps (y))$ denotes the group of diffeomorphisms that fixes some open neighborhood $D_\eps (y)$ pointwise, this completes the proof.
\end{proof}

  \subsection{Compatibility of the Pontrjagin-Thom maps}

 As described in \cite{madsenweiss}, \cite{GMTW},  the spaces  $\sgnm (K)$  form the spaces of morphisms in the topological cobordism category $\cc_2(K)$
 whose objects are nonnegative integers $n \geq 0$,  and whose morphisms $Mor_{\cc_2(K)}(n,m)$ are given by the disjoint union of the spaces  
  $$Mor_{\cc_2(K)}(n,m)  = \coprod_{g \geq 0} \sgnm (K),$$ except if $n=m$, in which case $Mor_{\cc_2(K)}(n,m)  = \left(\coprod_{g \geq 0} \crs_{g,n,n}(K)\right) \, \coprod id_n.$  
  
  One of the main theorems of \cite{madsenweiss}, \cite{GMTW} was the identification of the homotopy type of the classifying space of the cobordism category, $B\cc_2(K)$ (as well as other, more general, but similarly defined cobordism categories).  
  
  \bf Remark.  \rm In section 5 of \cite{madsenweiss} a category $\cc_{d, \theta}$ was defined, given a Serre fibration $\theta : B \to BO(d)$.  If $d=2$ and $\theta$ is the composition of the projection and the orientation cover,
  $$
  BSO(2) \times K \to BSO(2) \to BO(2),
  $$
  we have $B\cc_{2, \theta} \simeq B\cc_2(K)$.  The difference between $\cc_{2, \theta}$ and $\cc_2(K)$ is that the objects of the former are restricted to be a fixed union of circles in $\br^\infty$, while in $\cc_{2, \theta}$ the objects are arbitrary closed, oriented one-dimensional manifolds.  Using the covering isotopy theorem, it follows easily that the corresponding nerves are homotopy equivalent.

  \med
   The identification of the homotopy type of  $B\cc_2(K)$ used the Pontrjagin-Thom construction to define a functor to the path category of  the zero space of a certain Thom spectrum.  More specifically, let $\ock$ denote the path component of the basepoint
  in the zero space of the spectrum $\Sigma \bcp^\infty_{-1} \wedge K_+$.  Let $\cp (\ock)$ be its path category.  Namely, given any connected space $Y$, the path category $\cp (Y)$ is the topological category whose objects are the points of $Y$ (topologized as the space $Y$), and the morphisms between points $y_1$ and $y_2$ are the space of paths $\cp_{y_1, y_2}(Y) = \{(\alpha, r): \alpha : [0,r] \to Y \quad \text{is continuous, satisfying} \quad \alpha (0) = y_1, \, \alpha (r) = y_2 \}$.
  The following is standard:
  
  \begin{lemma} Given any connected space $Y$ there is a weak homotopy equivalence $B\cp (Y) \simeq Y$.
  \end{lemma}
  
  We record the standard fact that in any connected space $Y$ with basepoint $y_0$,  there is a homotopy equivalence of the path space with the loop space, $\cp_{y_1, y_2}(Y)  \simeq \Omega Y$. Such homotopy equivalences are given by choices of fixed paths $\gamma_1$ from $y_0$ to $y_1$, and $\gamma_2$ from $y_2$ to $y_0$.  The homotopy equivalence sends  a path  $\alpha $ from $y_1$ to $y_2$  to the glued path $\gamma_1 \circ \alpha \circ \gamma_2$ which is a loop at the basepoint $y_0 \in Y$.   
  
  In $\cite{madsenweiss}$ the Pontrjagin-Thom construction was used to produce a functor $\alpha : \cc_2(K) \to \cp (\ock)$ and the following was proved.
  \begin{theorem}\cite{madsenweiss} The functor $\alpha : \cc_2(K) \to \cp (\ock)$ induces a weak homotopy equivalence on the level of classifying spaces,
  $$
  \alpha : B \cc_2(K) \to B\cp (\ock) \simeq \ock.
  $$
  \end{theorem}
  
  Now for each $n\geq 0$ let $\gamma_n \in \ock$ be the image of the functor $\alpha$ on  $n  \in Ob \, \cc_2(K)$.  The functor $\alpha$ is ``pointed" in that $\gamma_0 \in \ock$ is the basepoint.    Now consider the map defined on the level of morphisms,
  $$
  \alpha_{n, m} : \coprod_{g \geq 0} \sgnm (K)  \to \cp_{\gamma_n, \gamma_m}  (\ock) \simeq \ockz.
  $$
  
  We consider the homotopy compatibility of these maps as $n$ and $m$ vary.  We focus our attention on the two cases  $\alpha_{0,1}$ and $\alpha_{0,0}$. 
  To compare them we consider a morphism $ (D^2, 1, e) \in \mathscr S_{0,1,0}(K) \subset Mor_{\cc_2(K)}(1, 0)$.  Here $D^2$  is embedded in $\br^\infty \times [0, 1]$ with boundary equal to $c_1 \subset \br^\infty \times \{0\}$.  The embedding is fixed, and has the property that its intersection with $\br^\infty \times [\frac{1}{2}, 1]$ is empty.  (In other words its image lies in $\br^\infty \times [0, \frac{1}{2})$.) The map $e : D^2 \to K$ is constant at the basepoint $x_0 \in K$.  
  
  Notice that composing with the morphism $(D^2, 1, e)$  defines a map
  $$
\kappa : \mathscr S_{g, 0,1}(K) \to \crs_g(K).
  $$
  This amounts to ``capping off the hole"   in a surface with one boundary component, and extending a map from that surface to the resulting closed surface
  by letting it be constant on the capping disk.
  
  Now let $\delta = \alpha_{1,0}(D^2, 1, e)$ be the image under the functor $\alpha$ of  $(D^2, 1, e)$, viewed as a morphism in $Mor_{\cc_2(K)}(1, 0)$.  The following compatibility theorem is now simply a result of the functoriality of $\alpha$.
  
  \begin{theorem}\label{compatible}  The following diagram commutes:
  $$\begin{CD}
  \mathscr S_{g, 0,1}(K)  @>\alpha_{0,1}  >> \cp_{\gamma_0, \gamma_1}(\ock) \\
  @V\kappa VV   @V\simeq V\delta V \\
  \mathscr S_g(K)   @>>\alpha_{0,0} > \cp_{\gamma_0, \gamma_0}(\ock) @>>=> \ockz.
  \end{CD}$$
  \end{theorem}
  where the right hand vertical map is concantentation with the fixed path $\delta$.   
   
\noindent
\bf Comments.  \rm
(1).  Notice that the map $\kappa$ in this theorem is homotopic to the projection map
 $$
 EDiff(F_{g, 1}, \p) \times_{Diff(F_{g, 1}, \p)} Map((F_{g, 1}, \p(F_{g,1})), (K, x_0)) \la  EDiff(F_g) \times_{Diff(F_g)} Map (F_{g} , K)
 $$
 defined by capping off the boundary of $F_{g,1}$ with a disk, and extending a map $f : (F_{g, 1}, \p(F_{g,1})) \to  (K, x_0)$ to the closed surface $F_g = F_{g,1} \cup D^2$, by defining it on $D^2$ to be constant at the basepoint $x_0 \in K$.   It therefore factors, up to homotopy, as the composition,
 $$
   \sgok  \xr{\iota} \sgo (K) \xr{p} \crs_g(K)
   $$
   where $\iota$ is as in Proposition \ref{markfib}, and $p: \sgo (K) \to \crs_g(K)$ forgets the marked point.   
    
 (2).  Similar compatibility results between the $\alpha_{n,m}$'s exist in general, by capping off various boundary circles.  We leave the formulation of these to the reader.
  
  \subsection{Stability with marked points}
  
  We now have the ingredients necessary to prove the main theorem of this section,  which gives a stability theorem for the moduli space of closed surfaces in a background space with marked points.   
  
  Consider the space $\sgo(K)$.  Let $\tilde \alpha : \sgo (K) \to  \ockz$ be the composition
\begin{equation}\label{alphao}
  \tilde \alpha : \sgo (K)  \xr{p} \crs_g(K) \xr{\alpha} \ockz.
\end{equation}
  We then define the map $\alpha^1 : \sgo (K) \to \ockz \times \bcp^\infty \times K
 $ to be the product  $\alpha^1 = \tilde \alpha \times \tau_v \times e$, with $\tau_v$ and $e$ as defined in (\ref{teevee}) and (\ref{ee}).
 
 \begin{theorem}\label{markedpoint}
 The map
 $$
 \alpha^1 : \sgo (K) \to \ockz \times \bcp^\infty \times K
 $$
 induces an isomorphism in integral homology, $H_q(-; \bz)$,  for $3q \leq 2g-2.$
 \end{theorem}
 \begin{proof}  By Propostion \ref{markfib} and Theorem \ref{compatible}, we have the following homotopy commutative diagram of homotopy fibration sequences:
$$
\begin{CD}
 \sgok  @>\iota >> \sgo (K) @>\tau_v \times e >> \bcp^\infty \times K \\
 @V\alpha_{0,1} VV  @V\alpha^1VV  @VV=V  \\
 \ockz   @>>> \ockz \times \bcp^\infty \times K  @>>>\bcp^\infty \times K
 \end{CD}
 $$
 Since the basespace of these fibrations, $\bcp^\infty \times K$, is simply connected, and since $\alpha_{0,1}$ induces a homology isomorphism
 in this range,  then $\alpha^1$ induces a homology isomorphism in this range.
 \end{proof}
 
 We remark that we can also consider moduli spaces of closed surfaces in a background space $K$ with $q$-marked points, $\crs^q_g(K)$.  Completely analogous arguments to the above prove
 the following.  We leave details to the interested reader:
 
 \begin{theorem}\label{manymarked}  There is a map
 $$
 \alpha^q : \crs^q_g(K)  \to \ockz \times (\bcp^\infty)^q \times K
 $$
 that induces an isomorphism in integral homology, $H_q(-; \bz)$,  for $3q \leq 2g-2.$
 \end{theorem}
 
  \section{Relation to the transfer}
  Consider the bundle
  $ F_g \to \sgo (K) \xr{p} \crgk$ described in Proposition \ref{markfib}.  The Becker-Gottlieb transfer map \cite{beckergottlieb} is a map of suspension spectra,
  $$
  t : \Sigma^\infty (\crgk_+)   \to \Sigma^\infty (\sgo (K)_+).
  $$
  Its induced map in integral cohomology, $t^*: H^*(\sgo (K)) \to H^*(\crgk)$ has the following  well-known properties:
  
  \begin{lemma}\label{trans} For $\beta \in H^*(\crgk)$ and $\eps \in H^*(\sgo (K))$,
  \begin{enumerate}
  \item $ t^*(p^*(\beta)\eps) = \beta\cdot t^*(\eps)$
  \item $t^*(1) = \chi (F_g) \in H^0(\crgk).$
  \end{enumerate}  Here  $\chi (F_g) = 2-2g$ is the Euler characteristic.
  \end{lemma}
  
  We will need to relate the transfer map with the Pontrjagin-Thom map $\alpha$ described above.  To do this we first need to consider the spectrum
  map
  $$
  \tilde w : \bcp^\infty_{-1} \to \Sigma^\infty (\bcp^\infty_+).
  $$
  This can be viewed as ``collapsing" the $-2$-dimensional sphere in $\bcp^\infty_{-1}$, but more precisely it is induced by
  map of Thom spectra,
  $$
  \tilde w :  \bcp^\infty_{-1} = (\bcp^\infty)^{-L} \to (\bcp^\infty)^{-L \oplus L} = \Sigma^\infty (\bcp^\infty_+).
  $$
  Here $L \to \bcp^\infty$ is the canonical oriented two dimensional bundle, and $-L$ is the corresponding virtual
  bundle given by its opposite.  The exponential notation $X^\zeta$ denotes the Thom spectrum of a virtual bundle $\zeta$ over a space $X$.  
  This map of Thom spectra is induced by the inclusion of virtual bundles, $-L \hk -L \oplus L$, where, of course, $-L \oplus L$ is the trivial zero dimensional virtual bundle.  
  
  Taking the smash product with the identity produces a similar map
  \begin{equation}\label{doubu}
  \tilde w(K) : \bcp^\infty_{-1} \wedge K_+  = (\bcp^\infty \times K)^{-L} \to (\bcp^\infty \times K)^{-L \oplus L} = \Sigma^\infty((\bcp^\infty \times K)_+).
  \end{equation}
  Here we are thinking of $L \to \bcp^\infty \times K$ as the pull-back of the bundle $L \to \bcp^\infty$ under the projection map $\bcp^\infty \times K \to \bcp^\infty$. 
  
  Consider the induced map on zero spaces,
  $$
  w(K) : \ockz \to \Omega^\infty \Sigma^\infty ((\bcp^\infty \times K)_+),
  $$ as well as the cohomology suspension map
  $$
  \sigma^* : H^*(\bcp^\infty \times K) \to H^*(\Omega^\infty \Sigma^\infty ((\bcp^\infty \times K)_+)).
  $$
  
  \med
  \begin{lemma}\label{lemone}  Let $\rho = \tau_v \times e : \sgo (K) \to \bcp^\infty \times K$.  Then if
  $\xi \in H^*( \bcp^\infty \times K)$,  
  $$
  t^*\rho^*(\xi) = \alpha^*w(K)^*\sigma^*(\xi) \in H^*(\crgk).
  $$
  \end{lemma}
  
  \med
  Before we prove this, we show how lemmas \ref{trans} and  \ref{lemone}  together imply the main theorem of this section:
  
  \med
  \begin{theorem}\label{main}
 $$ \alpha^* : H^*(\Omega^\infty_0(\bcp^\infty_{-1} \wedge K_+); \bq) \la H^*(\crgk; \bq)$$
 is an isomorphism for $3* \leq 2g-3.$
 \end{theorem}
 
 \med
 \begin{proof}  It follows from Lemma \ref{trans} that $t^* \circ p^* : H^*(\crgk) \to H^*(\crgk)$ is multiplication by $\chi (F_g) = 2-2g$.
 So with rational coefficients,  this composition is an isomorphism so long as $g \neq 1$.  (Notice the statement of the theorem is vacuous if $g= 1$, so we lose no generality in assuming $g \neq 1$.)  It follows that $p^*$ is injective, and $t^*$ is surjective in rational cohomology.
 
 In the stable range (dimensions less than or equal to $\frac{2}{3}g - 1$),  Theorem \ref{markedpoint} implies that  $(\alpha^1)^* : H^*(\ockz \times \bcp^\infty \times K) \to H^*( \sgo (K))$ is an isomorphism.  Now by the definition of $\alpha^1$   given in (\ref{alphao}) above, the following diagram   commutes:
  $$
  \begin{CD}
  \sgo (K)  @>\alpha^1 >>  \ockz \times \bcp^\infty \times K \\
  @V p VV   @VVproj. V \\
  \crgk @>> \alpha >    \ockz.
  \end{CD}
  $$
  Thus $p^* \circ \alpha^*$ is injective in this stable range, and hence so is $\alpha^*$.  It remains to show that $\alpha^*$ is surjective in this range.
  
  Let $\eta \in H^*(\Omega^\infty_0(\bcp^\infty_{-1} \wedge K_+) ; \bq)$ and $\xi \in H^*(\bcp^\infty \times K; \bq)$ be classes so that the sum of their dimensions is in the stable range.  Then by definition,  $(\alpha^1)^*(\eta \otimes \xi) =  p^*\alpha^*(\eta) \cdot \rho^*(\xi)$, and by Lemma \ref{trans}, $t^*(\alpha^1)^*(\eta \otimes \xi) =t^*p^* \alpha^*(\eta)\cdot t^*\rho^*(\xi)  = (2-2g) \alpha^*(\eta)\cdot t^*\rho^*(\xi)$. Lemma  \ref{lemone}  shows that $t^*\rho^*(\xi)$, and therefore $t^*(\alpha^1)^*(\eta \otimes \xi)$ belong to the image of $\alpha^*$.  Since $t^*$ is surjective and $(\alpha^1)^*$ is an isomorphism in this range, we conclude that $\alpha^*$ is surjective in this range.
  \end{proof}
  
  It remains to prove Lemma \ref{lemone}.   Consider again the fiber bundle, $F_g \to \sgo(K) \xr{p} \crgk.$  The transfer map, $t : \Sigma^\infty(\crgk_+) \to \Sigma^\infty (\sgo (K)_+)$ is defined to be the composition
  $$
  t : \Sigma^\infty(\crgk_+) \xr{\tilde t} \sgo (K)^{-\tau_v} \to  \sgo (K)^{-\tau_v \oplus \tau_v} = \Sigma^\infty (\sgo (K)_+)
  $$
  where $\tilde t : \Sigma^\infty(\crgk_+) \to \sgo (K)^{-\tau_v}$ is the Pontrjagin-Thom map, or ``pretransfer".  The bundle $\tau_v$ over  $\sgo (K)$ is the vertical tangent bundle, classified by $\tau_v : \sgo (K) \to \bcp^\infty$.   The map $\rho = \tau_v \times e : \sgo (K) \to \bcp^\infty \times K$  induces a map of Thom spectra,
  $$
  T(\rho) : \sgo (K)^{-\tau_v} \to (\bcp^\infty \times K)^{-L} = \bcp^\infty_{-1} \wedge K_+.
  $$
  As in \cite{madsenweiss}, \cite{GMTW},  the map of spectra $\tilde \alpha : \Sigma^\infty (\crgk_+) \to \bcp^\infty_{-1} \wedge K_+$ is defined to be the composition
  $$
   \tilde \alpha : \Sigma^\infty(\crgk_+) \xr{\tilde t} \sgo (K)^{-\tau_v} \xr{T\rho} (\bcp^\infty \times K)^{-\tau_v} = \bcp^\infty_{-1} \wedge K_+.
   $$ 
   $\alpha : \crgk \to \Omega^\infty(\bcp^\infty_{-1} \wedge K_+)$ is the adjoint of this map.  Using the definition of $\tilde w (K)$ given in (\ref{doubu}), we then have the following homotopy commutative diagram of spectra:
   $$
   \begin{CD}
   \Sigma^\infty(\crgk_+)    @>\tilde t >> \sgo (K)^{-\tau_v}   @>T(\rho) >> (\bcp^\infty \times K)^{-L}  @>\simeq >> \bcp^\infty_{-1} \wedge K_+ \\
   @V=VV    @VVV    @VVV  @VV\tilde w (K) V \\
    \Sigma^\infty(\crgk_+)    @> t >> \sgo (K)^{-\tau_v \oplus \tau_v} @>\rho >>   (\bcp \times K)^{-L\oplus L}   @>\simeq > > \Sigma^\infty ((\bcp^\infty \times K)_+)
  \end{CD}
  $$
  Since the top horizontal composition is $\tilde \alpha : \Sigma^\infty(\crgk_+) \to  \bcp^\infty_{-1} \wedge K_+$, we have that 
  $$
  \tilde w (K) \circ \tilde \alpha \simeq \rho \circ t : \Sigma^\infty(\crgk_+) \to \Sigma^\infty ((\bcp^\infty \times K)_+).
  $$ So in cohomology, $t^*\circ \rho^* = \tilde \alpha^* \circ \tilde w (K)^*$.  Combining this with the general relationship
  $$
   \tilde \alpha^* \circ \tilde w (K)^* = \alpha^* \circ w(K)^* \circ \sigma^*
   $$
   between maps of suspension spectra and their adjoints, completes the proof of the lemma.

   \section{A counter example to stability for mapping class groups for closed surfaces}
   The theorem below, due to S, Morita \cite{morita},  is a counter example to Theorem 0.4 of \cite{cohenmadsen} which postulates that $H_*(\Gamma_g,  V_g)$ is indepedent of the genus $g$, when the $\Gamma_g$-module $V_g  = V(F_g)$ is a coefficient satisfying certain degree conditions. We stress, however, that Theorem 0.4 remains true for surfaces with at least one boundary component:  $H_*(\Gamma_{g,n},   V(F_{g,n}))$ does have a stable range for $n>0$.  
   Here, as usual, $\Gamma_{g,n}$ is the mapping class group of orientation preserving diffeomorphisms of an oriented surface $F_g$ of genus $g$ with $n$ boundary components.  The modules $V (F_{g,n}) = H_1(F_{g,n})$ define a coefficient
   system satisfying the stated degree requirements, but as Morita's theorem below clearly implies, $H_*(\Gamma_g, H_1(F_g))$ clearly does not satisfy any stability property.  This example was pointed out to the authors by Johannes Ebert.
   
   \med
   \begin{theorem}(Morita) \cite{morita}   $$H_1(\Gamma_g; H_1 (F_g)) \cong \bz/2(g-1).$$
   \end{theorem}
   
   \med
   We give a proof of this theorem that is somewhat different in spirit from Morita's proof in \cite{morita}.  In fact the proof we present below is similar in spirit to the proof of our main Theorem \ref{main} above.
   
   \med
   Let $\gp = \pi_0(Diff (F_g, x_0))$ and $\gb = \pi_0(Diff (F_g, D^2)$ be mapping class groups, where $D^2 \subset F_g$ is a fixed, small disk, with $x_0 \in D^2$.  
   In $Diff (F_g, D^2)$ the diffeomorphisms fix a neighborhood of $D^2$ pointwise, and in $Diff (F_g, x_0)$ the diffeomorphisms are only required to fix the point $x_0$.  Consider the diagram
   $$
   \begin{CD}
   B\gb   @>>>B\gp @>D >>  BGL^+(2,\br) \\
   @V\kappa VV   @Vp VV   \\
   B\G_g  @>>=> B\G_b.
   \end{CD}
   $$
   The upper horizontal sequence is the homotopy fibration from Lemma \ref{fibration}.  The map $\kappa$  caps off the boundary, and $p$ forgets the marked point.    The improved Harer stability theorem \cite{harer} proved by Boldsen in \cite{boldsen},  asserts that
   $$
   \kappa_* :H_*(B\gb) \to H_*(B\G_g)
   $$
   is an isomorphism in the range $3* \leq 2g-2$, and surjective for $3* \leq 2g$.  Since $\kappa$ factors as the composition $\kappa : B\gb \to B\gp \xr{p} B\G_g$,     in this stable range  $p_*: H_*(B\gp) \to H_*(B\G_g)$ has a right inverse,   $S : H_*(B\G_g) \to H_*(B\gp)$  satisfying $p_*\circ S = id$.  
   In this range we also have the computation
   $$H_2(B\gp) \cong H_2(B\gb) \oplus H_2(BGL^+(2, \br)) \cong \bz \oplus \bz. $$  The Serre spectral sequence for the fibration $F_g \to B\gp \xr{p} B\G_g$  has
   $$
   E^2_{p,q} = H_p(\G_g; H_q(F_g)).
   $$
   The existence of the section $S :  H_*(B\G_g) \to H_*(B\gp)$  implies that there are no nontrivial differential emanating from the base line, $E^2_{*, 0}$ in the stable range.  Now we know that
   \begin{align}
   E^2_{0,1} &= H_0(\G_g; H_1(F_g))  \, = 0 \notag \\
   E^2_{1,1} &= H_1(\G_g; H_1(F_g)) \notag \\
   E^2_{0,2} &= H_0 (\G_g; H_2(F_g)) \, = H_2(F_g) \cong \bz.
   \notag 
   \end{align}
   So  the only possible differential in total degrees less than or equal to $2$ is
   $$
   d^2 : E^2_{2,1} \to  E^2_{0,2}.
   $$
   This leads to the exact diagram
   $$
   \begin{CD}
   0 @>>>  E^2_{0,2}/Im \, d^2   @>>>H_2(\gp )/Im \, S     @>>> H_1(\G_g; H_1(F_g)) \to 0 \\
   & & @AAA     @V\cong V D V  \\
   & & H_2(F_g) @>\tau_* >>  H_2(BGL^+(2, \br))  \, \cong \bz \\
    & & @AAd^2 A \\
  &&  H_2(\G_g; H_1(F_g)) 
    \end{CD}
    $$
    The groups $H_2(F_g)$ and $H_2(BGL^+(2, \br))$ are each isomorphic to $\bz$, and $\tau : F_g \to  BGL^+(2, \br))$ classifies its tangent bundle, and hence induces multiplication by the Euler characteristic $2-2g$ in $H_2$.    We claim that this implies that the differential $d^2 = 0$.   This is because, since $\tau_*$ is injective and $D$ is an isomorphism, the commutativity of the above diagram implies $H_2(F_g) \to  E^2_{0,2}/Im \, d^2$  is injective, which, by exactness implies $d^2$ is zero.   Since, as remarked above,  $H_2(F_g) \to  E^2_{0,2}$ is an isomorphism, the commutativity and exactness of this diagram imply
    $$
    H_1(\G_g; H_1(F_g)) \cong cok \, (\tau_*) \cong \bz/(2g-2).
    $$

 \end{document}